\documentclass[11pt,reqno]{amsart}

\usepackage[dvips]{graphicx}
\usepackage{amssymb, latexsym, amsmath, amscd, array, amsthm, epsfig, pstricks}
\usepackage[all]{xy}

\usepackage{hyperref}

\usepackage{breakurl}   

\newcommand{\HH}{{\mathbb H}}

\newcommand{\area}{{\rm area}}

\newcommand{\length}{{\rm length}}

\numberwithin{equation}{section}

\newtheorem{theorem}{Theorem}[section]
\newtheorem{proposition}[theorem]{Proposition}

\newtheorem{lemma}[theorem]{Lemma}

\theoremstyle{definition}
\newtheorem{definition}[theorem]{Definition}

\newtheorem{remark}[theorem]{Remark}

\long\def\forget#1\forgotten{} 

\makeatletter
\@namedef{subjclassname@2020}{\textup{2020} Mathematics Subject Classification}
\makeatother

\begin{document}

\title[Reverse isoperimetric inequalities in nonpositively curved
  cones] 
{Sharp reverse isoperimetric inequalities in nonpositively curved cones}

\author[M.\;Katz]{Mikhail G. Katz} \address{M.~Katz, Department of
Mathematics, Bar Ilan University, Ramat Gan 5290002 Israel}
\email{katzmik@math.biu.ac.il}

\author[S.\;Sabourau]{St\'ephane Sabourau}
\address{\parbox{\linewidth}{Univ Paris Est Creteil, CNRS, LAMA, F-94010 Creteil, France \\
Univ Gustave Eiffel, LAMA, F-77447 Marne-la-Vall\'ee, France}}
\email{stephane.sabourau@u-pec.fr}

\thanks{Partially supported by the ANR project Min-Max (ANR-19-CE40-0014).}



\begin{abstract}
We prove a pair of sharp reverse isoperimetric inequalities for
domains in nonpositively curved surfaces: (1) metric disks centered at
the vertex of a Euclidean cone of angle at least~$2\pi$ have minimal
area among all nonpositively curved disks of the same perimeter and
the same total curvature; (2) geodesic triangles in a Euclidean
(resp. hyperbolic) cone of angle at least~$2\pi$ have minimal area
among all nonpositively curved geodesic triangles (resp. all geodesic
triangles of curvature at most~$-1$) with the same side lengths and
angles.
\end{abstract}

\maketitle

\section{Introduction}

Isoperimetric inequalities provide upper bounds on the area of domains
in a surface with a fixed metric (typically of constant curvature) in
terms of their boundary length; see~\cite{BZ} for an account on this
classical subject.

Often the metric is fixed (flat, hyperbolic, or spherical), but there
are a few instances where isoperimetric inequalities hold for large
classes of metrics satisfying curvature bounds.  Thus, Andr\'e Weil
(\cite{weil}, 1926) developed such inequalities for nonpositively
curved planes, proving the Cartan-Hadamard isoperimetric conjecture in
dimension two.

In this article, we establish a pair of sharp reverse isoperimetric
inequalities providing \emph{lower bounds} on the area of some domains
in terms of their boundary length.

These geometric inequalities hold for nonpositively curved surfaces.
The optimal metrics for such reverse isoperimetric inequalities, as
well as the extremal domains, can be described in terms of the total
curvature of the domains under consideration.  The boundary cases of
equality in our optimal inequalities are attained by Euclidean cones
with nonpositive total curvature.  These inequalities provide bounds
that are stronger than Euclidean ones.

It seems that such geometric inequalities should have been known
already for some time, but we were unable to find them in the
literature.  A possible reason is that previous research focused more
on isoperimetric inequalities for special homogeneous metrics
regardless of the total curvature of the domains considered. \\

We will now present our two main results.
As motivation, note that the area of a Euclidean cone of angle less than~$2\pi$ can be easily decreased among nonnegatively curved metrics by smoothing off the tip of the cone.
Our first result shows that this is impossible while preserving the nonpositive curvature condition; see Theorem~\ref{theo:intro-isop2} for a more general statement.
Namely, one cannot decrease the area of a metric disk centered at the vertex of a Euclidean cone of angle at least~$2\pi$ among all nonpositively curved disks of the same perimeter and the same total curvature.
We will need the following definition.

\begin{definition}
Given a surface with a complete Riemannian metric with Gaussian
curvature function~$K$, we define the nonnegative and nonpositive
parts of~$K$ as
\[
K^+ = \max\{ K, 0 \} \quad \mbox{  and } \quad K^- = \max\{-K, 0 \}
\]
so that $K = K^+ - K^-$.
\end{definition}

Our convention is consistent with the corresponding definitions of
nonnegative and nonpositive parts of curvature measures introduced by
Yu.\;Burago in~\cite{Bur04}; see Definition~\ref{def:curv}.  We can
now state our first result.

\begin{theorem} 
\label{theo:intro-isop2}
Let~$M$ be a surface with a complete Riemannian metric.  Then every
disk~$D$ of radius~$R$ and boundary length~$L$ in~$M$ satisfies the
bound
\begin{equation}
\label{e11}
\area(D) \geq \area(\hat{D})
\end{equation}
where~$\hat{D}$ is the disk with the same boundary length~$L$ as~$D$,
centered at the vertex of the Euclidean cone with total curvature~$-
\int_D K^- \, dA$.
\end{theorem}

A formula for the area of~$\hat{D}$ is given in
Section~\ref{sec:disks}.

\medskip

Comparison geodesic triangles play an important role in nonpositive curvature geometry.  
For our second result, we consider a geodesic triangle of Gaussian curvature at most~$\lambda_0 \leq 0$ along with its comparison triangle (in the strong sense) in a cone~$\mathcal{C}_{\lambda_0}^\theta$ of constant nonpositive curvature~$\lambda_0$ with angle~$\theta$.
Here, a comparison triangle is taken in the following strong sense: both the angles and the side lengths of the two triangles are the same.  
The following theorem asserts that the area of the initial triangle is bounded from below by the area of its comparison triangle in the cone. 
(For simplicity, one can assume that $\lambda_0=0$.)

\begin{theorem} \label{theo:intro-triangle}
Let~$\Delta$ be a geodesic (two-dimensional) triangle in a surface with a complete Riemannian metric of Gaussian curvature~$K \leq \lambda_0$ for some constant~$\lambda_0 \leq 0$.
Suppose~$\bar{\Delta}$ is a geodesic (two-dimensional) triangle with the same side lengths and the same angles~$\alpha$,~$\beta$,~$\gamma$ as~$\Delta$ in the cone~$\mathcal{C}_{\lambda_0}^\theta$ of constant curvature~$\lambda_0$ with angle~$\theta = 3 \pi -(\alpha+\beta+\gamma)$.  
Then
\begin{equation*}
\area(\Delta) \geq \area(\bar{\Delta})
\end{equation*}
with equality if and only if~$\Delta$ is isometric to~$\bar{\Delta}$.
\end{theorem}

Our proof of Theorem~\ref{theo:intro-triangle} also shows that the area of a geodesic triangle of Gaussian curvature at most~$\lambda_0$ is bounded from below by the area of a comparison triangle having the same base length and the same adjacent angles in the plane~$\HH_{\lambda_0}$ of constant curvature~$\lambda_0 \leq 0$; see Proposition~\ref{prop:rauch}. \\



We would like to provide some context for our study of reverse
isoperimetric inequalities.  We found the geometric inequalities of
this paper while working on a proof that systolically extremal
nonpositively curved surfaces are flat with finitely many conical
singularities; see~\cite{KS}.  In that context, it was clear that it
should be imposssible to round off a conical singularity of angle
greater than~$2\pi$ in a nonsystolic region in order to decrease the
area (while keeping the nonpositive curvature condition) by any
cut-and-paste argument with metric disks of the same perimeter.  This
observation is formalized by our Theorem\;\ref{theo:intro-isop2}.
Theorem\;\ref{theo:intro-triangle} is a variation on this theme, while
trying to cut-and-paste triangles instead of disks.  Though we did not
use these reverse isoperimetric inequalities in our argument, they
confirmed our intuition that piecewise flat metrics with conical
singularities should play a role in extremal systolic geometry through
their local extremal features.

\medskip\noindent \emph{Acknowledgment.}  The second author would like
to thank the Fields Institute and the Department of Mathematics at the
University of Toronto for their hospitality while this work was
completed.

\section{Reverse isoperimetric inequality for metric disks} \label{sec:disks}

The reverse isoperimetric inequality for metric disks established in this section shows the optimality of the nonpositively curved Euclidean cones for the area with respect to compact deformations keeping the same curvature sign. \\

Before proving this result, we need to extend the notion of curvature to singular spaces.

\begin{definition} \label{def:curv}
Associated to a surface~$M$ with a complete Riemannian metric~$g$ of
Gaussian curvature~$K$ is the \emph{curvature measure}~$K \, dA$,
where $dA$ is the area measure of~$M$.  The notion of curvature
measure extends to piecewise flat surfaces with conical singularities
(and more generally to Alexandrov surfaces), where it is denoted
by~$\omega$.  The curvature measure~$\omega$ is a signed measure which
can be decomposed as $\omega = \omega^+ - \omega^-$, where $\omega^+$
and~$\omega^-$ are the nonnegative and nonpositive parts of~$\omega$,
and are both nonnegative measures.  For smooth metrics, we have
$\omega^{\pm} = K^{\pm} \, dA$.
\end{definition}

We refer to~\cite{AZ} and~\cite{res-book} for a precise definition of
the curvature measure; see also~\cite{tro} for a modern exposition on
Alexandrov surfaces.

We will make use of the following result on bi-Lipschitz metric
approximation, announced by Reshetnyak~\cite{Res59} and proved by
Yu.~Burago~\cite[Lemma~6]{Bur04}, in the more general setting of
Alexandrov surfaces; see also~\cite[Theorem~3.1.1]{BZ}, \cite{tro}
and~\cite[\S3]{KS}.

\begin{proposition}[{See~\cite{Res59} and~\cite[Lemma~6]{Bur04}}] \label{prop:approx}
Let $M$ be a compact surface (possibly with boundary) with a Riemannian metric.
Then there is a sequence~$M_i$ of piecewise flat surfaces with conical singularities converging to~$M$ in the Lipschitz topology such that the nonnegative and nonpositive parts~$\omega^{\pm}_i$ of the curvature measure~$\omega_i$ of~$M_i$ weakly converge to their counterparts~$\omega^{\pm}$ for the curvature measure~$\omega$ on~$M$.
%
\end{proposition}

Consider a sequence~$(g_i)$ of piecewise flat metrics with conical
singularities approximating a given complete Riemannian metric~$g$ on
a surface for every compact domain as in
Proposition~\ref{prop:approx}.  Denote by~$M$ the surface with the
complete Riemannian metric~$g$, and by~$M_i$ the same surface with the
piecewise flat metric~$g_i$.

\begin{proposition} \label{prop:D_i}
Fix $p \in M$.  Let $D \subseteq M$ and $D_i \subseteq M_i$ be the
disks of radius~$R$ centered at the same point~$p$.  
Then
\begin{enumerate}
\item the area of the symmetric difference~$D \Delta D_i$ tends to
zero (for every area measure).  That is, $| D \Delta D_i | \to 0$;
\label{item1}
\item  $\lim \area(D_i,g_i) = \area(D,g)$; \label{item2}
\item $\lim \omega^{\pm}_i(D_i) = \omega^{\pm}(D)$; \label{item3}
\item $\liminf \length(\partial D_i,g_i) \geq \length(\partial D,g)$. \label{item4}
\end{enumerate}
\end{proposition}

\begin{proof}
\eqref{item1} By bilipschitz convergence of the metrics, the symmetric
difference~$D \Delta D_i$ is contained in the $\varepsilon$-tubular
neighborhood~$U_\varepsilon(\partial D)$ of~$\partial D$ for $i$ large
enough.  Since $\partial D$ is $1$-rectifiable, the area of this
tubular neighborhood tends to zero (see~\cite[Theorem~3.2.39]{fed}),
and the result follows.

\eqref{item2} By bilipschitz convergence of the metrics, we have the inclusions
\[
D - \varepsilon \subseteq D_i \subseteq D + \varepsilon
\]
for $i$ large enough, where $D \pm \varepsilon \subseteq M$ are the balls of radius~$R \pm \varepsilon$ centered at~$p$.
We also have weak convergence of the area measures.
Taking the area with respect to~$g_i$ in the previous double inclusion between $D_i$ and~$D \pm \varepsilon$, and using the weak convergence of the area measure, we obtain
\[
\area(D -\varepsilon,g) - \varepsilon \leq \area(D_i,g_i) \leq \area(D+\varepsilon,g) + \varepsilon
\]
for $i$ large enough.  Since the area of the $\varepsilon$-tubular
neigborhood~$U_\varepsilon(\partial D)$ of~$\partial D$ tends to zero,
the result is immediate.

\eqref{item3} As in the proof of item~\eqref{item2}, using the weak
convergence of the curvature measure instead of the area measures, we
obtain
\[
\omega^{\pm}(D -\varepsilon) - \varepsilon \leq \omega^{\pm}_i(D_i) \leq \omega^{\pm}(D+\varepsilon) + \varepsilon
\]
for $i$ large enough.  Since the curvature measure
$\omega^{\pm} = K^{\pm} \, dg$ is absolutely continuous with
respect to the area measure, the curvature measure of the
$\varepsilon$-tubular neighborhood~$U_\varepsilon(\partial D)$
of~$\partial D$ tends to zero.  Hence the result.

\eqref{item4} The flat distance between the one-cycles~$\partial D$ and~$\partial D_i$ is bounded by the mass of the $2$-current defined as the difference~$D-D_i$, see~\cite{fed} for precise definitions.
This mass is equal to the area of~$D \Delta D_i$.
Thus, by~\eqref{item1}, the sequence~$\partial D_i$ converges to~$\partial D$ in the flat topology.
The desired result follows from the lower semicontinuity of the mass (here, the length); see~\cite{fed}.
\end{proof}

We can now proceed to the proof of the following theorem.  Recall that
a disk of radius~$R$ centered at the vertex of a Euclidean cone of
angle~$\theta$ has perimeter~$L=\theta R$ and area $A=\frac{\theta}{2}
R^2 = \frac{L^2}{2\theta}$.  Thus, we have
\begin{equation} 
\label{eq:cone}
A = \frac{L^2}{2(2\pi- \mathcal{K})}
\end{equation}
where $\mathcal{K} = 2\pi -\theta<0$ is the total curvature of the
Euclidean cone.

\medskip

With this formula, Theorem~\ref{theo:intro-isop2} can be restated as follows.

\begin{theorem} \label{theo:isop2}
Let~$M$ be a surface with a complete Riemannian metric.
Then every disk~$D$ of radius~$R$ of boundary length~$L$ in~$M$ satisfies
\[
\area(D) \geq \frac{L^2}{4\pi + 2\mathcal{K}_D^-}
\]
where~$\mathcal{K}_D^- = \int_D K^- \, dA \geq 0$ is the total mass of the nonpositive part of the curvature measure of~$D$.
\end{theorem}


\begin{proof}
Proposition~\ref{prop:D_i} shows that it is sufficient to prove
Theorem~\ref{theo:isop2} for piecewise flat metrics with conical
singularities.  The desired result for Riemannian metrics will follow
by piecewise flat metric approximation; see
Proposition~\ref{prop:approx}.  This enables us to avoid regularity
issues.


Let~$p$ be the center of the disk~$D \subseteq M$ of radius~$R$.
Consider the (closed) disk~$D_t$ and the circle~$\mathcal{C}_t$ of
radius~$R-t$ centered at~$p$.  Note that~$D_0=D$
and~$\mathcal{C}_0=\partial D$.  Since the metric on~$M$ is piecewise
flat, the circle~$\mathcal{C}_t$ is a piecewise smooth curve (possibly with several connected components)
bounding the metric disk~$D_t$ in~$M$.  Moreover, the length
function~$t \mapsto L(\mathcal{C}_t)$ is differentiable except for a
finite number of values of~$t$, and its derivative is given by the
first variation formula; see~\cite[Lemma~3.2.3]{BZ}.  Namely, as long
as~$D_t$ is nonempty, we have
\[
L'(\mathcal{C}_t) = - \int_{\mathcal{C}_t} \kappa(s) \, ds - S_t
\]
for almost every~$t$, where~$\kappa$ is the geodesic curvature of the curve~$\mathcal{C}_t$ and~$S_t$ is the sum of the angular difference of the tangent vectors at the corner points of~$\mathcal{C}_t$.

By the Gauss--Bonnet formula for polyhedral metrics
(see~\cite[Theorem~5.3.2]{res-book}), we derive
\[
L'(\mathcal{C}_t) = - 2 \pi \chi(D_t) + \omega(D_t).
\]
Since $D_t$ is a connected region with boundary, its Euler characteristic~$\chi(D_t)$ is at most~$1$, that is, $\chi(D_t) \leq 1$.
Since $D_t \subseteq D$ and $\omega = \omega^+ - \omega^-$, we have
\[
\omega(D_t) \geq - \omega^-(D).
\]
Combining these two bounds, we deduce that
\[
L'(\mathcal{C}_t) \geq - 2 \pi - \omega^-(D).
\]
Integrating this relation leads to
\begin{equation} 
\label{eq:LCt2}
L(\mathcal{C}_t) \geq L(\partial D) - \left( 2 \pi + \omega^-(D) \right) \, t.
\end{equation}
In particular, the domain~$D_t$ is nonempty for every~$t<t_0$, where
\[
t_0 = \frac{L(\partial D)}{2 \pi + \omega^-(D)}.
\]
By the coarea formula, integrating the inequality~\eqref{eq:LCt2}
between~$0$ and~$t_0$, we obtain
\[
\area(D) \geq L(\partial D) \, t_0 - \frac{1}{2} \, \left( 2 \pi + \omega^-(D) \right) t_0^2.
\]
In other words, we have
\[
\area(D) \geq \frac{L(\partial D)^2}{2 \left( 2 \pi + \omega^-(D) \right)}
\]
where the right-hand side represents the area of the disk centered at
the vertex of the Euclidean cone with total curvature~$\mathcal{K} = -
\omega^-(D)$ and with the same boundary length~$L(\partial D)$ as~$D$;
see formula~\eqref{eq:cone}.
\end{proof}

\section{Area comparison for triangles with the same base}

In order to prove Theorem~\ref{theo:intro-triangle}, we will need the following result, which may be of independent interest.  
This result provides a lower bound on the area of a geodesic triangle of Gaussian curvature at most~$\lambda_0$, in terms of the area of a comparison triangle having the same base length with the same adjacent angles in the plane~$\HH_{\lambda_0}$ of constant curvature~$\lambda_0 \leq 0$.

\begin{proposition} \label{prop:rauch}
Let~$\Delta$ be a geodesic (two-dimensional) triangle with vertices~$A$,~$B$,~$C$ in a surface~$M$ with a complete Riemannian metric of Gaussian curvature~$K \leq \lambda_0$ for some constant~$\lambda_0 \leq 0$.
Let~$\bar{\Delta}$ be a geodesic (two-dimensional) triangle with distinct vertices~$\bar{A}$,~$\bar{B}$,~$\bar{C}$ in the plane~$\HH_{\lambda_0}$ of constant curvature~$\lambda_0$ such that
\begin{itemize}
\item the sides~$\bar{A} \bar{B}$ and~$AB$ have the same length;
\item the angles at~$\bar{A}$ and~$A$ are the same;
\item the angles at~$\bar{B}$ and~$B$ are the same.
\end{itemize}
Then
\begin{equation} \label{eq:rauch}
\area(\Delta) \geq \area(\bar{\Delta})
\end{equation}
with equality if and only if~$\Delta$ is isometric to~$\bar{\Delta}$.
\end{proposition}

We first establish the proposition when one of the angles at~$A$ or~$B$, say~$B$, is greater or equal to~$\frac{\pi}{2}$; see Lemma~\ref{lem:baby}.
We will then derive the general result from this particular case.

\begin{lemma} \label{lem:baby}
Let $\Delta$ and~$\bar{\Delta}$ be as in Proposition~\ref{prop:rauch}.
Let $\alpha$, $\beta$, $\gamma$ be the angles of~$\Delta$ at~$A$, $B$, $C$.
Suppose that $\beta \geq \frac{\pi}{2}$.
Then
\begin{equation*}
\area(\Delta) \geq \area(\bar{\Delta})
\end{equation*}
with equality if and only if~$\Delta$ is isometric to~$\bar{\Delta}$.
\end{lemma}

\begin{proof}
Let~$\hat{\Delta}$ be a geodesic triangle of~$\HH_{\lambda_0}$ with the same side lengths as~$\Delta$.
Denote by~$\hat{A}$,~$\hat{B}$,~$\hat{C}$ the vertices
of~$\hat{\Delta}$ and by~$\hat{\alpha}$,~$\hat{\beta}$,~$\hat{\gamma}$
their angles.  We can (and will) assume that~$\hat{A} = \bar{A}$
and~$\hat{B}=\bar{B}$.  By~\cite[Proposition~II.1.7.(4)]{BH}, we have
\begin{equation}\label{eq:ab}
\begin{cases}
\hat{\alpha} \geq \alpha \\
\hat{\beta} \geq \beta.
\end{cases}
\end{equation}
In particular,~$\hat{\beta} \geq \frac{\pi}{2}$.  Since
the sum of the angles of a geodesic triangle in a nonpositively curved
surface is at most~$\pi$; see~\cite[Proposition~II.1.7.(4)]{BH}, we
derive that~$\hat{\gamma} \leq \frac{\pi}{2}$.

The geodesic triangle of~$\HH_{\lambda_0}$ with side~$\bar{A}\bar{B}$, and angles~$\alpha$
and~$\beta$ at~$\bar{A}$ and~$\bar{B}$ can be isometrically identified
to~$\bar{\Delta}$.  It follows from the relations~\eqref{eq:ab} that
the triangle~$\bar{\Delta}$ lies in~$\hat{\Delta}$; see
Figure~\ref{fig:ABC}.  Since~$\hat{\gamma} \leq \frac{\pi}{2}$, we
deduce that
\[
|\bar{A}\bar{C}| \leq |\hat{A}\hat{C}| = |AC|.
\]
This relation holds for any point~$D$ lying in the segment~$BC$ by replacing~$\hat{C}$ with~$\hat{D}$ and~$\bar{C}$ with~$\bar{D}$ (note that~$\bar{D}$ lies in the segment~$\bar{B}\bar{C}$).

\begin{figure}[htbp!] 
\vspace{1.7cm}
\def\svgwidth{5cm}
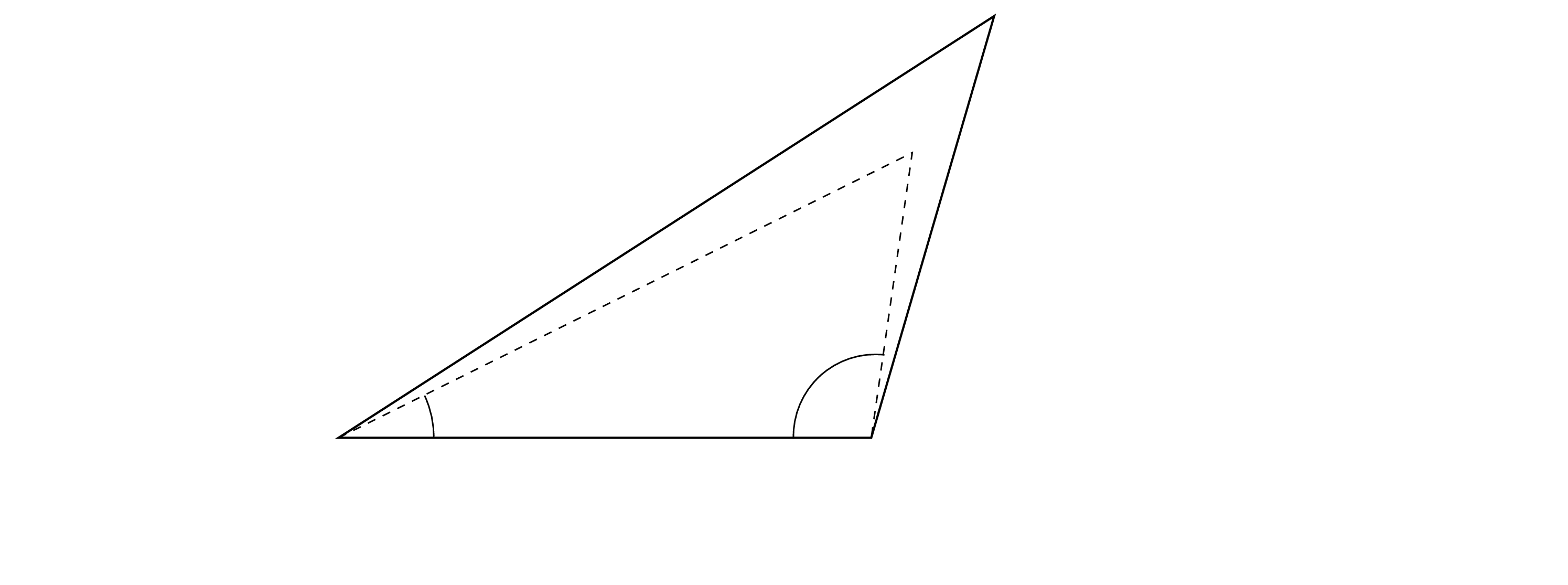 
\vspace{0.2cm}
\caption{The comparison triangles~$\bar{\Delta}$ and~$\hat{\Delta}$} \label{fig:ABC}
\end{figure}


Consider the two exponential maps $\exp_A:T_A M \to M$ and $\exp_{\bar{A}}: T_{\bar{A}} \HH_{\lambda_0} \to \HH_{\lambda_0}$.
Define $\sigma:\HH_{\lambda_0} \to M$ as
\[
\sigma = \exp_A \circ I \circ \exp_{\bar{A}}^{-1}
\]
where $I:T_{\bar{A}} \HH_{\lambda_0} \to T_A M$ is a linear isometry.
We can choose~$I$ so that the map~$\sigma$ takes the geodesic rays~$[\bar{A} \bar{B})$ and~$[\bar{A} \bar{C})$ to the geodesic rays~$[AB)$ and~$[AC)$.
By construction, the map~$\sigma$ sends every segment~$\bar{A} \bar{D}$ joining~$\bar{A}$ to a point~$\bar{D}$ in the opposite side~$\bar{B}\bar{C}$ of~$\bar{\Delta}$ to a subarc of the geodesic arc~$AD$.
In particular, the map~$\sigma$ sends~$\bar{\Delta}$ into~$\Delta$.

Since $K \leq \lambda_0$, the Rauch theorem implies that the map~$\sigma$ is distance-nondecreasing;
see~\cite[Theorem~6.5.4]{BBI}.  Therefore,
\begin{equation*} \label{eq:temp}
\area(\bar{\Delta}) \leq \area(\Delta)
\end{equation*}
with equality if and only if the map~$\sigma$ is an isometry between~$\bar{\Delta}$ and~$\Delta$.
\end{proof}

We can now derive Proposition~\ref{prop:rauch} from the previous lemma.

\begin{proof}[Proof of Proposition~\ref{prop:rauch}]
Let~$\alpha$,~$\beta$,~$\gamma$ be the angles of~$\Delta$ at~$A$,~$B$,~$C$.
The cases where $\alpha \geq \frac{\pi}{2}$ or $\beta \geq \frac{\pi}{2}$ are covered by Lemma~\ref{lem:baby}.
Thus, we can assume that~$\alpha < \frac{\pi}{2}$ and~$\beta < \frac{\pi}{2}$.
This implies that the projection~$H$ of~$C$ to the segment~$AB$ strictly lies between~$A$ and~$B$.
Furthermore, both angles~$\measuredangle{AHC}$ and~$\measuredangle{BHC}$ are equal to~$\frac{\pi}{2}$.
Thus, the height~$CH$ decomposes~$\Delta$ into two right triangles~$\Delta'$ and~$\Delta''$.
Denote by~$\gamma'$ and~$\gamma''$ the angles of~$\Delta'$ and~$\Delta''$ at~$C$.
Observe that~$\gamma = \gamma' + \gamma''$.
Since the angles of~$\Delta'$ and~$\Delta''$ at~$H$ are right, we can
apply Lemma~\ref{lem:baby} to the triangles~$\Delta'=AH'C'$
and~$\Delta''=BH''C''$, where~$H=H'=H''$ and~$C=C'=C''$.  
Thus, the areas of the triangles~$\Delta'$ and~$\Delta''$ are bounded from below by the areas of the right triangles~$\bar{\Delta}'=\bar{A}\bar{H}'\bar{C}'$ and~$\bar{\Delta}''=\bar{B} \bar{H}'' \bar{C}''$ in~$\HH_{\lambda_0}$.  
We can glue together these two right triangles of~$\HH_{\lambda_0}$ along their
sides~$\bar{H}'\bar{C}'$ and~$\bar{H}''\bar{C}''$ so that~$\bar{H}'$
and~$\bar{H}''$ coincide.  Since the angles at~$\bar{H}'$
and~$\bar{H}''$ are right, the two segments~$\bar{A} \bar{H}'$
and~$\bar{H}'' \bar{B}$ form a long segment~$\bar{A} \bar{B}$ of the
same length as~$AB$; see Figure~\ref{fig:right-triangles}.

\medskip

If the heights~$\bar{H}' \bar{C}'$ and~$\bar{H}'' \bar{C}''$ have the same length, the union of these two triangles form a large triangle which satisfies the same geometric features as~$\bar{\Delta}$ and so can be identified with~$\bar{\Delta}$.

\medskip

If one of these heights, say~$\bar{H}' \bar{C}'$, is shorter than the
other, we extend the hypothenuse~$\bar{A}\bar{C}'$ until it intersects
the other hypothenuse~$\bar{B} \bar{C}''$ at some point~$\bar{O}$; see
Figure~\ref{fig:right-triangles}.  As previously, the
triangle~$\bar{A} \bar{O} \bar{B}$ can be identified
with~$\bar{\Delta}$.  Moreover, the area of this triangle is bounded
by the sum of the two triangles~$\bar{\Delta}'$ and~$\bar{\Delta}''$.

\begin{figure}[htb] 
\vspace{-0.5cm}
\def\svgwidth{4cm}
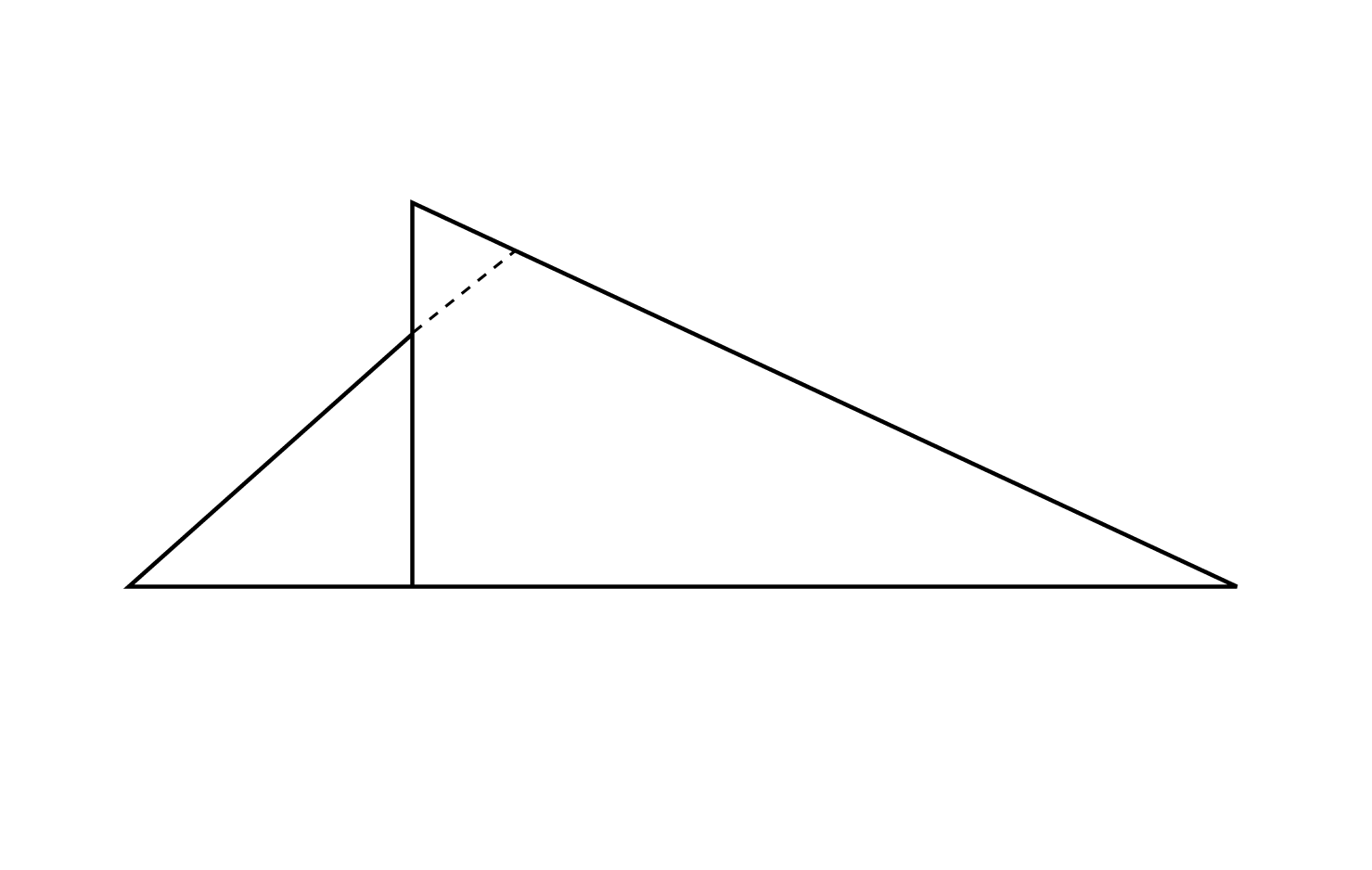
\vspace{0.3cm}
\caption{The two triangles~$\bar{\Delta}'$ and~$\bar{\Delta}''$} \label{fig:right-triangles}
\end{figure}

\medskip

In either case, we have
\[
\area(\bar{\Delta}) \leq \area(\bar{\Delta}') + \area(\bar{\Delta}'') \leq \area(\Delta)
\]
with equality if and only if~$\Delta$ is isometric to~$\bar{\Delta}$. 
\end{proof}

\section{Area comparison for triangles with the same side lengths and angles}

We can now prove our second main result.

\begin{theorem}  \label{theo:areaT}
Let~$\Delta$ be a geodesic (two-dimensional) triangle in a surface with a complete Riemannian metric of Gaussian curvature~$K \leq \lambda_0$ for some constant~$\lambda_0 \leq 0$.
Suppose~$\bar{\Delta}$ is a geodesic (two-dimensional) triangle with the same side lengths and the same angles~$\alpha$,~$\beta$,~$\gamma$ as~$\Delta$ in the cone~$\mathcal{C}_{\lambda_0}^\theta$ of constant curvature~$\lambda_0$ with angle~$\theta = 3 \pi -(\alpha+\beta+\gamma)$.  
Then
\begin{equation} \label{eq:areaT}
\area(\Delta) \geq \area(\bar{\Delta})
\end{equation}
with equality if and only if~$\Delta$ is isometric to~$\bar{\Delta}$.
\end{theorem}

\begin{remark}
Unlike the situation with Proposition~\ref{prop:rauch}, the comparison
triangle~$\bar{\Delta}$ in~$\mathcal{C}_{\lambda_0}^\theta$ in Theorem~\ref{theo:areaT} do not necessarily
exist in general.
%
\end{remark}

\begin{proof}[Proof of Theorem~\ref{theo:areaT}]
Let~$A$,~$B$ and~$C$ be the vertices of~$\Delta$, with angle $\alpha$
at the vertex~$A$, angle $\beta$ at~$B$, and angle $\gamma$ at~$C$.
Denote by~$\bar{A}$,~$\bar{B}$,~$\bar{C}$ the corresponding vertices
of~$\bar{\Delta}$.
%

\medskip

By~\cite[Proposition~II.1.7.(4)]{BH}, the angle~$\theta$ of the
conical singularity of the cone is at least~$2 \pi$.  Moreover, we can
assume that~$\theta > 2\pi$ and that the conical singularity lies
in~$\bar{\Delta}$.  Otherwise, the sum~$\alpha+ \beta +\gamma$ of the
angles of~$\Delta$ and~$\bar{\Delta}$ would be equal to~$\pi$ and so
the triangle~$\Delta$ would be flat isometric to~$\bar{\Delta}$
by~\cite[Proposition~II.2.9]{BH}.

\medskip 

The geodesic rays joining the vertices of~$\bar{\Delta}$ to its conical singularity decompose each angle around the vertices of~$\bar{\Delta}$ into two angles.
In particular, the angle~$\alpha$ splits into two angles~$\alpha'$ and~$\alpha''$, where~$\alpha = \alpha' + \alpha''$.
The same holds with~$\beta$ and~$\gamma$.
Observe that these geodesic rays decompose~$\bar{\Delta}$ into three small triangles.

\medskip

We would like to carry out a similar construction for~$\Delta$, except
that there is no conical singularity to rely on.  Instead, we consider
the geodesic rays of~$\Delta$ emanating from the vertices of~$\Delta$
and splitting each angle as in~$\bar{\Delta}$.  These three geodesic
rays do not necessarily intersect at a single point as
in~$\bar{\Delta}$.  Nevertheless, they decompose~$\Delta$ into three
triangles~$\Delta_a$,~$\Delta_b$,~$\Delta_c$ as in
Figure~\ref{fig:three-triangles}.  Recall that two geodesic rays in a
nonpositively curved surface intersect at most once.  There may be a
small triangular region
\[
X = \Delta \setminus (\Delta_a \cup \Delta_b \cup \Delta_c)
\]
lying in~$\Delta$ which is not
covered by the triangles~$\Delta_a$,~$\Delta_b$,~$\Delta_c$.

\begin{figure}[htb] 
\vspace{-0.1cm}
\def\svgwidth{4.5cm} 
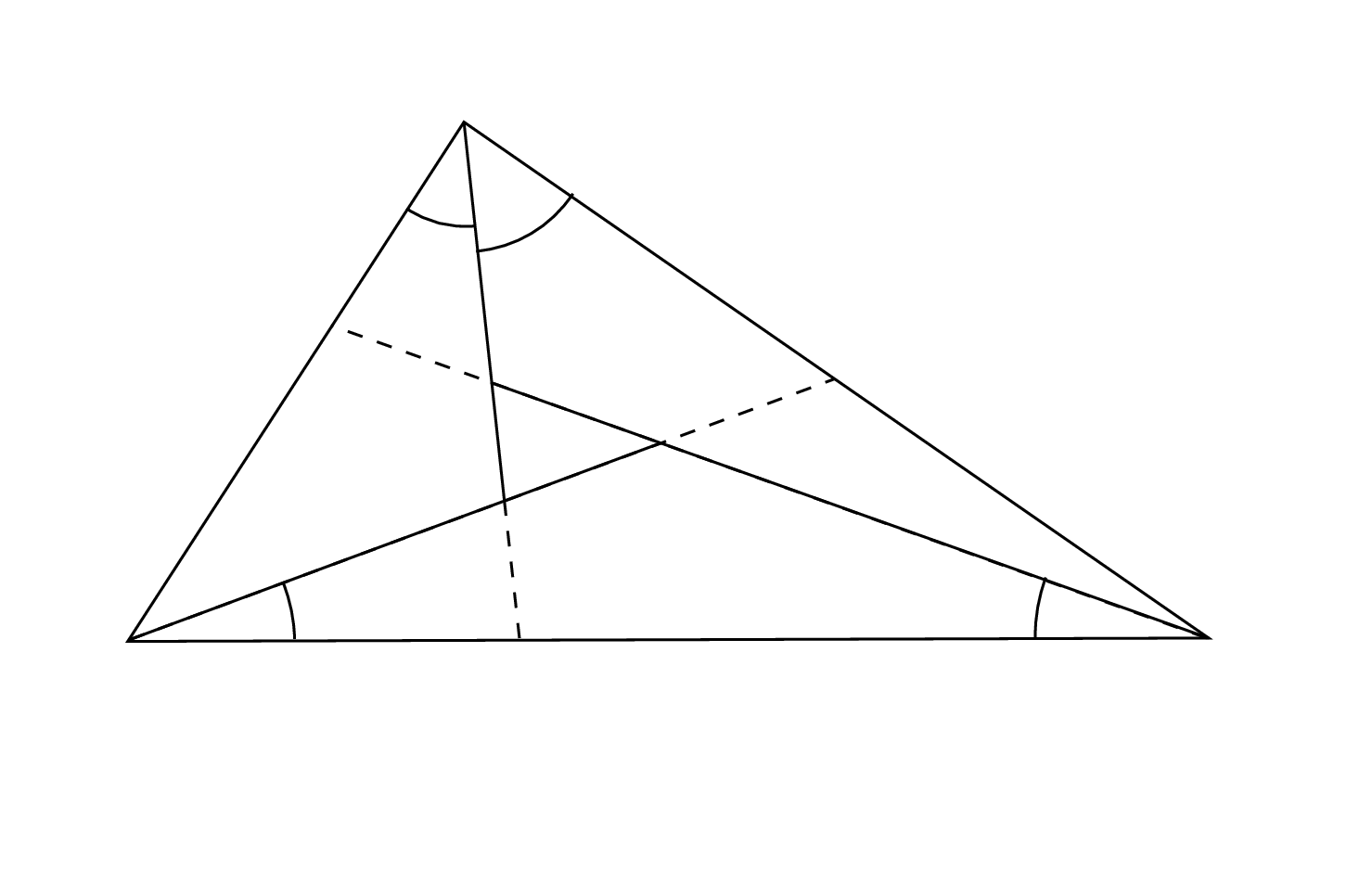 
\vspace{0.2cm}
\caption{Decomposition of~$\Delta$ into triangles} 
\label{fig:three-triangles}
\end{figure}

Let~$A'$ be the vertex of~$\Delta_a$ different from~$B$ and~$C$.
Denote by~$\bar{\Delta}_a$ the triangle of~$\mathcal{C}_{\lambda_0}^\theta$ with vertices~$\bar{A}'$,~$\bar{B}$,~$\bar{C}$ such that
\begin{itemize}
\item the sides~$\bar{B} \bar{C}$ and~$BC$ have the same length;
\item the angles at~$\bar{B}$ and~$B$ are the same (equal to~$\beta''$);
\item the angles at~$\bar{C}$ and~$C$ are the same (equal to~$\gamma'$).
\end{itemize}
Similarly, we define~$\bar{\Delta}_b$ and~$\bar{\Delta}_c$.  By
construction, the three triangles~$\bar{\Delta}_a$,~$\bar{\Delta}_b$
and~$\bar{\Delta}_c$ are isometric to the three smaller triangles
forming~$\bar{\Delta}$ and delimited by the geodesic segments joining
the conical singularity of~$\bar{\Delta}$ to its vertices.  By
Proposition~\ref{prop:rauch}, we have
\[
\area(\bar{\Delta}_a) \leq \area(\Delta_a)
\]
with equality if and only if~$\Delta_a$ is isometric
to~$\bar{\Delta}_a$.  The same holds with~$\bar{\Delta}_b$
and~$\bar{\Delta}_c$.  Since the triangle~$\bar{\Delta}$ is
partitioned into~$\bar{\Delta}_a$,~$\bar{\Delta}_b$
and~$\bar{\Delta}_c$, we derive that
\[
\begin{aligned}
\area(\bar{\Delta}) & = \area(\bar\Delta_a) + \area(\bar\Delta_b) +
\area(\bar\Delta_c) \\&\leq \area(\Delta_a) + \area(\Delta_b) +
\area(\Delta_c) \\&=\area(\Delta)-\area(X)\\& \leq\area(\Delta)
\end{aligned}
\]
with equality if and only if~$\Delta$ is isometric to~$\bar{\Delta}$.
\end{proof}

\begin{remark}
It follows from the proof of Theorem~\ref{theo:areaT} that the
difference between $\area(\Delta)$ and~$\area(\bar{\Delta})$ is
bounded from below by the area of the small triangle $X= \Delta
\setminus (\Delta_a \cup \Delta_b \cup \Delta_c)$ lying in~$\Delta$.
\end{remark}

\end{document}